\documentclass[12pt]{amsart}
\usepackage{amsmath, amsthm, amssymb, amsfonts}
\usepackage{mathrsfs}
\usepackage{amscd, tikz-cd}
\usepackage{hyperref}
\usepackage{ulem}
\usepackage{xcolor}
\usepackage{braket}
\usepackage{fullpage}
\usepackage{graphicx}%
\usepackage{multirow}%
\usepackage{amsmath,amssymb,amsfonts}%
\usepackage{amsthm}
\usepackage{mathrsfs}%
\usepackage[title]{appendix}%
\usepackage{xcolor}%
\usepackage{textcomp}%
\usepackage{manyfoot}%
\usepackage{booktabs}%
\usepackage{algorithm}%
\usepackage{algorithmicx}%
\usepackage{algpseudocode}%
\usepackage{listings}%
\usepackage{mathtools}
\usepackage{stackengine}
\usepackage[all]{xy}
\stackMath
\usepackage{mathrsfs}
\usepackage{hyperref}
\usepackage[T1]{fontenc} 
\makeatletter
\newcommand*{\da@rightarrow}{\mathchar"0\hexnumber@\symAMSa 4B }
\newcommand*{\da@leftarrow}{\mathchar"0\hexnumber@\symAMSa 4C }
\usepackage{geometry}
\usepackage{geometry}
\newtheorem{thm}{Theorem}[section]
\newtheorem{prop}[thm]{Proposition}
\newtheorem{lem}[thm]{Lemma}
\newtheorem{cor}[thm]{Corollary}
\newtheorem{conj}[thm]{Conjecture}
\newtheorem{ques}[thm]{Question}

\theoremstyle{definition}
\newtheorem{defn}[thm]{Definition}

\theoremstyle{remark}
\newtheorem{remark}[thm]{Remark}

\numberwithin{equation}{section}

\usepackage{enumerate}
\usepackage{verbatim}
\usepackage{enumitem}
\usepackage{relsize} 
\usepackage[bbgreekl]{mathbbol} 
\usepackage{amsfonts} 
\DeclareSymbolFontAlphabet{\mathbb}{AMSb} 
\DeclareSymbolFontAlphabet{\mathbbl}{bbold}

\usepackage[alphabetic]{amsrefs}

\newcommand{\N}{\mathbb{N}}
\newcommand{\F}{\mathbb{F}}
\newcommand{\Z}{\mathbb{Z}}
\newcommand{\Q}{\mathbb{Q}}
\newcommand{\R}{\mathbb{R}}
\newcommand{\C}{\mathbb{C}}

\newcommand{\ok}{\mathcal{O}_{K}}
\newcommand{\OK}{\mathcal{O}_{K}}
\newcommand{\M}{\mathcal{M}}
\newcommand{\G}{\mathcal{G}}

\newcommand{\Img}{\operatorname{Im}}

\newcommand{\Gal}{\operatorname{Gal}}

\newcommand{\bE}{\breve{E}}
\newcommand{\bF}{\breve{\mathcal{F}}}

\newcommand{\Spec}{\mathrm{Spec}}
\newcommand{\Sp}{\mathrm{Sp}}

\newcommand{\defeq}{\vcentcolon=}

\renewcommand{\Im}{\operatorname{Im}}

\begin{document}

\title{Towards the $p$-adic nowhere density conjecture of Hecke Orbits }
\author{Yu Fu}
\begin{abstract}{
 We propose a conjecture on the $p$-adic nowhere density of the Hecke orbit of subvarieties of Hodge type Shimura varieties. By investigating the monodromy of $p$-adic Galois representations associated with points on such Shimura varieties, we prove that the locus in a formal $\ok$-neighborhood of a mod $p$ point that has large monodromy is open dense, where $K$ is a totally ramified finite extension of $\breve{\Q}_p.$
}
\end{abstract}
\maketitle

\section{Introduction}
Let $\mathcal{S}h_{k}$ be a Shimura variety over a field $k$. The \textit{Hecke orbit} $\mathcal{H}(x)$ of a point $x \in \mathcal{S}h_{k}$ is the set of points $y \in \mathcal{S}h_{k}$ related to $x$ by isogenies. Let $k$ be an algebraically closed field of positive characteristic $p$. If the degrees of the isogenies in the previous definition are not divisible by $p$, then we call the set the \textit{prime-to-$p$ Hecke orbit} and denote it by $\mathcal{H}^{(p)}(x)$. The Hecke orbit conjecture predicts that Hecke symmetries characterize the central foliation in $\mathcal{S}h_{k}$ and has the original formulation by Oort: 
\begin{conj}
Given any $\overline{\mathbb{F}}_p$-point $x=\left[\left(A_x, \lambda_x\right)\right]$ in the moduli space $\mathcal{A}_g$ of $g$ dimensional principally polarized abelian varieties in characteristic $p>0$, $\mathcal{H}(x)$ is Zariski dense in the Newton polygon stratum of $\mathcal{A}_g$ which contains $x$.
\end{conj}
    The prediction can be generalized to the special fibers of Shimura varieties at primes of good reduction: any prime-to-$p$ Hecke orbit is dense in the central leaf containing it. This was recently proved in various settings by a series of papers; see \cite{Yu06}, \cite{Chai05}, \cite{Cha06}, \cite{xiao20}, \cite{HoftenOrdinary}, \cite{DvV}. 
    
    However, it is noticeable that the behavior of Hecke correspondences induced by isogenies between abelian varieties in characteristic $p$ and $p$-adically is significantly different from the behavior in characteristic zero and under the topology induced by Archimedean valuations. For example, one can show that $\mathcal{H}(x)$ is dense in the Archimedean topology for $x \in \mathcal{S}h_{\C}$ \cite[Theorem 6.1]{EY03}. Also, there are works on the Archimedean equidistribution properties of Hecke orbits, we refer the reader to \cite{Duk88,CU,Zha05}.  

On the other hand, it is possible to prove that the Hecke orbit of a point $x \in \mathcal{S}h(\C_p)$ on a PEL type Shimura variety with good reduction at $p$ is nowhere dense in the $p$-adic topology, by showing the nowhere density of the Picard number jumping locus in a family of abelian varieties with good reduction at $p$, which implies that points with extra endomorphisms are nowhere dense; see\cite{MP}. 

\vspace{1em}
Previously, we have focused on the theory of Hecke orbits of points on Shimura varieties. Inspired by the work of Maulik-Poonen, it is natural to ask for a higher-dimensional analogue question, in particular in the setting of $p$-adic topology.
\begin{ques}
     Can we describe the geometry of the (irreducible components of) positive-dimensional subvarieties of a Shimura variety that are stable under the Hecke operators? 
\end{ques}

We propose the following conjecture on the $p$-adic nowhere density of the Hecke orbit:

 \begin{conj}\label{conj}
   Let $Z$ be a positive-dimensional subvariety of a Shimura variety $\mathcal{S}h$. The Hecke orbit of $Z$ is $p$-adically nowhere dense in $\mathcal{S}h(\C_p)$.
   \end{conj}

\vspace{1em}
The conjecture is still wild open, even in the case where $Z$ is a positive-dimensional Shimura subvariety. The results in this direction have various possible applications. For example, applying the conjecture to the Torelli locus, which is a subvariety of $\mathcal{A}_g$, and leveraging the fact that $\overline{\mathbb{Q}}$ is dense in $\mathbb{C}_p$, one obtains a new proof of the existence of abelian varieties over $\overline{\mathbb{Q}}$ that are not isogenous to Jacobians. Similarly, there are potential applications to Zilber-Pink-type questions. For example, it provides mutual evidence of the existence of abelian varieties over $\overline{\Q}$ that are not quotients of low-dimensional Jacobians, which was proved in the very recent work of Tsimerman \cite{Tsi}. Infact, if there exists an abelian variety $A$ of dimension $g \ge 3$, a curve $C$ of genus $g+1$, and a non-trivial map $\phi: \operatorname{Jac}(C) \to A$. Then the Jacobian $\operatorname{Jac}
(C)$ is isogenous to $A \times E$ for some elliptic curve $E$. Tsimerman's theorem indicates that the Hecke orbit of the Torelli locus is ``small'' so that it is disjoint from $\mathcal{A}_{g} \times \mathcal{A}_{1}$. In addition, the conjecture implies, as a direct corollary, that the set of CM points in $\mathcal{S}h_{\C_p}$ is closed in the rigid analytic topology. 

\vspace{1em}

We investigate the conjecture by studying the monodromy of $p$-adic Galois representations attached to mod $p$ points on a Shimura variety. Our main result is the following formal density theorem that describes the locus in the deformation spaces whose points have large $p$-adic monodromy. 

\vspace{1em}
To state our results, we fix some notation. Let $p \geq 3$ be a prime. Let $(G, X)$ be a Hodge-type Shimura datum such that $G$ is unramified at $p$ defining the Shimura variety $\mathcal{S}h$. Let $(G, [b], \mu)$ be the local Shimura datum associated with a point $x \in \mathcal{S}h(\overline{\mathbb{F}}_p)$, let $E \subset \overline{\Q}_p$ be the field of definition of $\mu$, and let $\breve{E}$ be the completion of the maximal unramified extension of $E$ at the prime corresponding to the choice of the embedding $\overline{\Q} \hookrightarrow \overline{\Q}_p$ (see Section \ref{localdatum} for details). Let $A_x$ be the abelian variety associated with $x$ and denote by $A_x[p^\infty]$ the $p$-divisible group of $A_x.$ The formal neighborhood $\mathcal{S}^{x}$ of $x$ consists of $\C_p$-deformations of $A_{x}[p^{\infty}]$, due to the fact that deforming an abelian variety is equivalent to deforming its $p$-divisible group. Let $K$ be a finite extension of $\breve{E}$, and let $y$ be a $\ok$-point in $\mathcal{S}^x$. There is a Galois representation  $$\rho_y: \Gal(\bar{\breve{E}}/K) \to V_p(A_y[p^\infty])$$ on the rational Tate module $$V_p(A_y[p^\infty]):=V_p\left(A_y[p^\infty] \times_{\mathcal{O}_K} K\right)=T_p(A_y[p^\infty]) \otimes_{\mathbb{Z}_p} \mathbb{Q}_p$$ 
 of the $p$-divisible group $A_y[p^\infty]$, where $A_y$ is the abelian variety associated with $y$.
\begin{defn}\label{defnlarge monodromy}
We say that $y$ has \textit{large monodromy} if the image of $\rho_y$ contains an open subgroup of the derived group $G^{der}$. 
\end{defn}
We prove the following theorem.
\begin{thm}\label{padic}

	Let $x \in \mathcal{S}h(\overline{\mathbb{F}}_p)$ be an integral point and let $(G, [b], \mu)$ be the local Hodge-Shimura datum associated with $x$, such that the Newton polygon and the Hodge polygon associated with $[b]$ and $\mu$ do not touch outside endpoints (i.e. satisfy the HN-irreducibility). Let $\mathcal{S}^{K,x}$ be the set of $\OK$-points of $\mathcal{S}^{x}$ each of these special fibers is isomorphic to $A_x[p^{\infty}]$. Let $\mathcal{S}_{\operatorname{max}}^{K,x}$ be the subset of points that have large monodromy. Then
	 $$\mathcal{S}_{\operatorname{max}}^{K,x} \subset \mathcal{S}^{K,x}$$ 
	 is $p$-adically open dense. 

\end{thm}

This generalizes the work of Maulik and Poonen in the PEL case. We note that their proof use ``$p$-adic Lefschetz $(1,1)$ theorem'' of Berthelot and Ogus, combined with p-adic analysis. Our approach is completely different; the key ingredients of our proof are the classification of finite flat group schemes by Breuil \cite{BR00} and uniformization theorems of Hodge type Shimura varieties. 

Also, we note that, we cannot prove the entire conjecture using the current strategy. Indeed, a crucial step towards a complete proof of Conjecture \ref{conj} is to prove that Theorem \ref{padic} holds not only for deformations over a fixed finite extension of $\breve{E}$, but for all the $\C_p$-deformations. Define the \textit{locus of small monodromy} as the set of $\mathbb{C}_p$ points that are algebraic points without large monodromy, or transcendental but are limit points of algebraic points with small monodromy. In \cite[Lemma 2.2.1]{kis}, Kisin proved that the Galois action on the first $p$-adic \'etale cohomology groups of the deformations factor through $G(\mathbb{Q}_p)$. Therefore, if $x \in \mathcal{S}h(\C_{p})$ is in a proper Shimura subvariety of $\mathcal{S}h$, then it belongs to the small monodromy locus. If one can prove that the locus of large monodromy on $\mathcal{S}h(\C_p)$ is open dense, then the conjecture in the case where $Z$ is a Shimura subvariety follows. A $\C_p$-neighborhood of an algebraic point with large monodromy may contain infinitely many points defined over field extensions with arbitrary large wild ramifications; this leads to the obstruction to the generalization of Theorem \ref{padic} to $\C_p$ points. In particular, suppose that the monodromy group of $y \in \mathcal{S}^{K,x}$ contains an open subgroup of $G^{der}$, say $H_n$(see Proposition \ref{open neighborhoods} for a definition), and let $z \in \mathcal{S}^{K',k}$ be a point defined over a wild ramified extension $K'/K$ such that $y \equiv z \bmod p^{n+1}$. Applying Theorem \ref{open neighborhoods} to $y'\coloneq y \times_K K'$ , the monodromy of $y'$ will ``shrink'', so does the open neighborhood of large monodromy. Therefore, we need to impose some kind of ``minimality'' condition on the wild ramification of the field of definition of those algebraic points, which is not possible even in the $p$-power Hecke orbit of a point. As a corollary, we can say a word about the prime-to-$p$ Hecke orbit of a point. 

\begin{cor}
 Let $x \in \mathcal{S}h_{\C_p} $ be a point with good reduction and let $A_x[p^{\infty}]$ be $p$-divisible group associated with $\bar{x}$ where the condition on the polygons is satisfied. Then $\mathcal{H}^{(p)}(x)$ is $p$-adic nowhere dense. \end{cor}

\section{A congruence condition for large monodromy}

The goal of this section is to prove a sufficient condition for the $p$-adic Galois associated to a fixed lift of a mod p point having a large monodromy to have a (p-adic) neighborhood of lifts with a large monodromy. The Serre-Tate theorem states that the deformation theory of an abelian variety is completely determined by the deformation theory of its associated p-divisible group, meaning that deforming an abelian variety is equivalent to deforming its p-divisible group. Therefore, we do this by studying the corresponding $p$-divisible groups. The main result of this section is the following. 

\begin{thm}\label{congruence condition}
	Let $\mathcal{G}^{\prime} \in \mathcal{S}^{K,x}$ be the p-divisible group associated to a moduli point $y \in \mathcal{S}^{K,x}$. Suppose that the Galois image of $\mathcal{G}^{\prime} \in \mathcal{S}^{K}$ contains a $p$-adic open subgroup of $G^{der}$ in the form $H_n=G^{der}(\Z_p)\cap K_n$, where
    $$K_n=\{M \in \Sp_{2g}(\Z_{p}) \mid M \equiv I \text{ mod } p^{n} \}$$ for some fixed $n \ge 1$. Then for all $m \ge n+2$, the $p$-adic neighborhoods of $G^{\prime}$
	$$U_{m}:=\{\G \in \mathcal{S}^{K,x} \mid \G \equiv \G^{\prime} \text{ mod } p^{m}\}$$
 consists of $p$-divisible groups with large monodromy.
\end{thm}

\subsection{Classification of finite flat group schemes over $p$-adic rings}\label{section classifiction}
 Our first step is to study the $p$-divisible groups attached to the $\ok$-lifts and determine conditions under which their $p^n$-torsions coincide as group schemes over $\ok$. We give a brief introduction to the theory of Breuil modules and the resulting classification of finite flat $p$-groups, which we will use in the proof of Theorem \ref{congruence condition}. For further details, we refer the reader to Chapters 2 and 4 of \cite{BR00}.
 
We fix some conventions that will be used later. Let $p$ be a prime number, $k$ be a perfect field of characteristic $p > 0$ and $W = W (k)$ be the ring of Witt vectors with coefficients in $k$. Let $K_{0}=Frac(W)$ and let $K$ be a totally ramified finite extension of $K_{0}$ of degree $e \ge 1$. Let $\mathcal{O}_{K}$ be the ring of integers of $K$ and let $\pi $ be a uniformizer of $\mathcal{O}_{K}$. Let $u$ be an indeterminate, $S$ be the p-adic completion of the envelope to the divided powers of $W[u]$ with respect to the ideal generated by the minimal polynomial of $\pi$ over $K_{0}$ and Fil$^{1}S$ be the p-adic completion of this ideal with divided powers. 

\vspace{1em}
We call a group scheme a \textit{p-group}, if it is a commutative group killed by a power of $p$. Denote by $\left(p-Gr / \mathcal{O}_K\right)$ the category of finite and flat $p$-groups over $\operatorname{Spec}\left(\mathcal{O}_K\right)$. Let $\operatorname{Spf}\left(\mathcal{O}_K\right)_{s y n}$ be the category of $p$-adic syntomatic formal schemes on $\operatorname{Spf}\left(\mathcal{O}_K\right)$ equipped with the Grothendieck topology generated by the surjective families of syntomic morphisms. Define $\left(Ab/ \mathcal{O}_K\right)$ as the abelian category of commutative group schemes over $\operatorname{Spf}\left(\mathcal{O}_K\right)_{s y n}$.

 The classical crystalline Dieudonn\'e theory gives a classification of finite flat group schemes  over $\mathcal{O}_K $. In \cite{BR00}, Breuil introduced the concept of strongly divisible modules and gave a classification of finite flat group schemes killed by $p$ for any $p \ge 3$, in terms of filtered modules with an action by the divided Frobenous \cite[Theorem 1.2]{BR00}. This classification generalizes to finite flat $p$-groups by associating to such group schemes certain tractable and concrete linear algebraic objects $(\mathcal{M} , \operatorname{Fil}^{1} \mathcal{M} , \phi_{1} )$.

\begin{thm}{\cite[Theorem 1.3]{BR00}}\label{Breuil}
Suppose $p \ge 3$, there is an anti-equivalence of categories between the category of finite and flat $p$-groups on $\OK$ whose kernel of the multiplication by $p^{n}$ is still flat for all $n$ and the category of objects $(\mathcal{M} , \operatorname{Fil}^{1} \mathcal{M} , \phi_{1})$ where $\M$ is an $S$-module of the form  $\oplus_{i \in I}S/p^{n_{i}}S$ for finite $I$ and $n_{i} \in \N^{+}$, $\operatorname{Fil}^{1}\M$ a sub $S$-module containing $\operatorname{Fil}^{1}S\cdot \M$ and $\phi_{1}$ a semi-linear mapping from $\operatorname{Fil}^{1}\M$ to $\M$ whose image generates $\M$ over $S$.
\end{thm}

Concretely, there is a filtered object associate to any group $G$ of $(p-Gr/\OK)$
$\operatorname{Mod(G)}$ defined by:
     \begin{itemize}
     \item
     	$\mathcal{M}:=\operatorname{Mod}(G)=\operatorname{Hom}_{\left(\mathrm{Ab} / \mathcal{O}_{K}\right)}\left(G, \mathcal{O}_{\infty, \pi}^{\mathrm{cris}}\right)$
     	
     	\item
$\operatorname{Fil}^1 \mathcal{M}:=\operatorname{Fil}^{1} \operatorname{Mod}(G)=\operatorname{Hom}_{\left(\mathrm{Ab} / \mathcal{O}_{K}\right)}\left(G, \mathcal{J}_{\infty, \pi}^{\text {cris }}\right)$
\item

$\phi_{1}: \operatorname{Fil}^{1} \mathcal{M} \rightarrow \mathcal{M} \text { is induced by } \phi_{1}: \mathcal{J}_{\infty, \pi}^{\text{cris}} \rightarrow \mathcal{O}_{\infty, \pi}^{\text{cris}}$.

     \end{itemize}

  By construction, $\operatorname{Mod}(G)$ is in the category $'(\operatorname{Mod}/S)$ and the morphisms in $'(\operatorname{Mod}/S)$ are the $S$-linear morphisms that preserve $\operatorname{Fil}^1 \mathcal{M}$ and commute with $\phi_1$ (\cite[2.1.1]{BR00}). Moreover, we can restrict the functor $\operatorname{Mod}$ to the full subcategory of $(p-Gr/\OK)$ of group schemes whose kernel of the multiplication by $p^{n}$ map is flat for all $n$, and deduce Theorem \ref{Breuil} from \cite[Theorem 4.2.1.6]{BR00}. We will use this concrete description of $\operatorname{Mod}(G)$ later in our proof.

\begin{remark}
 Recall that $([\mathrm{BBM}, 3.3.8])$ a \textit{truncated Barsotti-Tate group of level} $n \geq 2$ is an object of $\left(p-G r / \mathcal{O}_K\right)$ killed by $p^n$ such that for all $r \in\{0, \ldots, n\}$, $$ \operatorname{Ker}\left(p^{n-r}: G \rightarrow G\right)=\Im (p^r:G \to G).$$ For example, if $A$ is an abelian variety over $\mathcal{O}_K$, we have $ A\left[p^n\right]=\operatorname{Ker}(p^n: A \rightarrow A)$ which is a truncated Barsotti-Tate group of level $n$ for all $n \leq 2$. This implies that the kernel of the multiplication by $p^r$ on $G$ is flat for all $r$ and that the bi-algebra of $G$ is a free $\mathcal{O}_K$-module of rank $p^{nd}$ for the height $d$ of $G$, so that Theorem \ref{Breuil} applies.
\end{remark}

\subsection{The sheaf $\mathcal{O}^{cris}_{\infty, \pi}$ and $\mathcal{J}^{cris}_{\infty , \pi}$}\label{section crystalline sheaves}
Although we do not delve into the technical details of the proof of Theorem \ref{Breuil}, in this section, we recall some necessary results on the properties of sheaves on the syntomic site of $\operatorname{Spec}\left(\mathcal{O}_K\right)$. For further details, we refer the reader to \cite[Section 2.3]{BR00}.

For a fixed $n$, we denote $S_n=S / p^n$ and $E_{n}=\operatorname{Spec}\left(S_{n}\right)$. Let $ \operatorname{Spec}\left(\mathcal{O}_{K} / p^{n}\right) \hookrightarrow E_{n}$ be the divided power thickening of $\operatorname{Spec}\left(\mathcal{O}_K/ p^{n}\right)$. Let $\mathfrak{X}$ be a syntomic p-adic formal scheme over $\operatorname{Spf}\left(\mathcal{O}_K\right)$. Let $X$ be a scheme over $\operatorname{Spec}\left(\mathcal{O}_{K} / p^{n}\right)$. One can define the small crystalline site $\left(X / E_{n }\right)_{\text {cris }}$ and the sheaves $\mathcal{J}_{X / E_{n}} \subset \mathcal{O}_{X / E_{n}}$ on it, as in \cite[III.1.1]{Be}.

One defines presheaves $\mathcal{J}_{n, \pi}^{\text{cris}} \subset \mathcal{O}_{n, \pi}^{\text {cris }}$ on $\operatorname{Spf}\left(\mathcal{O}_{K}\right)_{\text{syn}}$ by letting:

\begin{align}
\mathcal{O}_{n, \pi}^{\text{cris}}(\mathfrak{X}) &=H_{\text{cris}}^{0}\left(\left(\mathfrak{X}_{n} / E_{n}\right)_{\text{cris}}, \mathcal{O}_{\mathfrak{X}_{n}/ E_{n}} \right) \\
\mathcal{J}_{n, \pi}^{\text{cris}}(\mathfrak{X}) &=H_{\text{cris}}^{0}\left(\left(\mathfrak{X}_{n} / E_{n}\right)_{\text{cris}}, \mathcal{J}_{\mathfrak{X}_{n} / E_{n}}\right)
\end{align}

where $\mathfrak{X}_{n}=\mathfrak{X} \times \mathbb{Z} / p^{n} \mathbb{Z}$. The argument in \cite[Section 1.3]{FM} can be applied to show that these are sheaves on $\operatorname{Spf}\left(\mathcal{O}_{K}\right)_{\text{syn}}$. Among the results, we point out the following significant lemma.

\begin{lem}{\cite[Corollary 2.3.3]{BR00}}\label{Breuil 2.3.3}
	The sheaves $\mathcal{O}_{n, \pi}^{\text {cris }}$ and $\mathfrak{J}_{n, \pi}^{\text {cris }}$ are flat on $W_{n }$ and we have exact sequences in $\left(A b / \mathcal{O}_{K}\right): 0 \rightarrow \mathcal{O}_{i, \pi}^{cris} \stackrel {p^{n}}{\rightarrow} \mathcal{O}_{n+i, \pi}^{c r i s} \rightarrow \mathcal{O}_{n, \pi}^{\text {cris } } \rightarrow 0$ and $0 \rightarrow \mathfrak{J}_{i, \pi}^{\text{cris }} \stackrel{p^{n}}{\rightarrow}$ $\mathfrak{J}_{n+i, \pi}^{ c r i s} \rightarrow \mathfrak{J}_{n, \pi}^{c r i s} \rightarrow 0.$

\end{lem}

\vspace{1em}

We summarize sections \ref{section classifiction} and \ref{section crystalline sheaves} using the following lemma, which classifies the $p$-power torsions of $p$-divisible groups over $\ok$ by their values modulo $p^{n+1}$. 

\begin{lem}\label{breuil lemma}
Let $\mathcal{G}$ and $\mathcal{G}^{\prime}$ be $p$-divisible groups over $\ok$ such that $\mathcal{G} \equiv \mathcal{G}^{\prime} \text{ mod } p^{n+1}$ for some positive integer $n$, let $\mathcal{G}[p^{n}]$ and $\mathcal{G}^{\prime}[p^{n}]$ be the $p^n$ torsion of $\mathcal{G}$ and $\mathcal{G}^{\prime}$, respectively. Then we have $\mathcal{G}[p^{n}] \simeq \mathcal{G}^{\prime}[p^{n}]$ as group schemes over $\ok.$   
\end{lem}

\begin{remark}
	Here, by saying that $\mathcal{G} \equiv \mathcal{G}^{\prime} \text{ mod } p^{n}$, we mean that $\mathcal{G}$ is isomorphic to $\mathcal{G}^{\prime}$ after base change to the ring $\ok/p^{n}$. 
	
\end{remark}
 \begin{proof} Fix a positive integer $n$. For any positive integer $m$, let $$G_m=\mathcal{G}[p^n] \times_{\Spec(\ok)} \Spec(\ok/p^m)$$ and $$G^{\prime}_m=\mathcal{G}^{\prime}[p^n] \times_{\Spec(\ok)} \Spec(\ok/p^m)$$ be the base change to $\ok/p^{m}$. By assumption, $G_{n+1}$ is isomorphic to $G^{\prime}_{n+1}$ as group schemes over $\ok/p^{n+1}$. The previous argument associates a triple to any $p^n$-torsion group scheme $G$ as follows
 $$G \rightsquigarrow (\mathcal{M}, \operatorname{Fil}^1 \mathcal{M}, \phi_1).$$
 Breuil's theorem \ref{Breuil} implies that such $G$ is classified by the triple. We claim that this classification depends on the reduction of the $p^n$-torsion subgroup modulo $p^{n+1}.$     
 \begin{itemize}
     \item The $S_n$ module $\mathcal{M}=\operatorname{Hom}_{\left(\mathrm{Ab} / \mathcal{O}_{K}\right)}\left(G, \mathcal{O}_{n, \pi}^{\mathrm{cris}}\right)$ depends only on the reduction modulo $p$ of the group scheme, so does $\operatorname{Fil}^1 \mathcal{M}=\operatorname{Hom}_{\left(\mathrm{Ab} / \mathcal{O}_{K}\right)}\left(G, \mathcal{J}_{n, \pi}^{\mathrm{cris}}\right)$. This is due to \cite[Corollary 3.2.2]{BR00}, we have $$\operatorname{Hom}_{S / E_n}\left(\underline{G}_1, \mathcal{O}_{S / E_n}\right) \simeq \operatorname{Hom}_{\left(A b / \mathcal{O}_K\right)}\left(G, \mathcal{O}_{n, \pi}^{c r i s}\right)$$ where $\underline{G}_1$ is the sheaf on $\left(S / E_n\right)_{CRIS}$ defined by $$\underline{G}_1(U, T)=G_1(U)=\operatorname{Hom}_S\left(U, G_1\right)$$ (resp. $\operatorname{Fil}^1 \mathcal{M}$ and $\mathcal{J}_{n, \pi}^{cris}$). 
     \item We claim that the \textit{divided Frobenius} $\phi_1 =\frac{\phi}{p}$, is defined by the reduction modulo $p^{n+1}$ of the group scheme. Denote by 
\begin{align}
      \mathcal{O}_{n, \pi}^{\text {cris }}(G) &\defeq H_{\text {cris }}^0\left(\left(G_n / E_n\right)_{\text {cris }}, \mathcal{O}_{G_n / E_n}\right)  \\
       \mathcal{J}_{n, \pi}^{\text {cris }}(G)&\defeq H_{\text {cris }}^0\left(\left(G_n / E_n\right)_{\text{cris}}, \mathcal{J}_{G_n / E_n}\right).
     \end{align}
There are naturally embeddings
 \begin{align}
     \operatorname{Hom}_{Ab / \ok}\left(G, \mathcal{O}_{n, \pi}^{\operatorname{cris}}\right) &\hookrightarrow \mathcal{O}_{n, \pi}^{\text {cris }}(G) \\
     \operatorname{Hom}_{A b / \ok}\left(G, \mathcal{J}_{n, \pi}^{\operatorname{cris}}\right) &\hookrightarrow \mathcal{J}_{n, \pi}^{\text {cris }}(G).
 \end{align}
 By the definitions in \cite[2.1.1]{BR00}, our claim is equivalent to saying that when $G_{n+1} \cong G^{\prime}_{n+1}$ over $\ok/p^{n+1}$, the following diagram commutes
    $$ \begin{tikzcd}
\mathcal{J}_{n, \pi}^{cris}(G[p^{n}]) \arrow[r, "\phi_1"] \arrow[d]
    & \mathcal{O}_{n, \pi}^{cris}(G[p^{n}]) \arrow[d, "\psi" ] \\
    \mathcal{J}_{n, \pi}^{cris}(G^{\prime}[p^{n}]) \arrow[r, "\phi_1"] & \mathcal{O}_{n, \pi}^{cris}(G^{\prime}[p^{n}])
     \end{tikzcd}. $$

We prove this by diagram chasing in the following diagrams. Fix the isomorphism $\psi: G_{n+1} \xrightarrow{\sim} G^{\prime}_{n+1}$, which also induces $G_{n} \xrightarrow{\sim} G^{\prime}_{n}.$ We denote this map also by $\psi$ for ease of notation.

$$
\xymatrix@C-=0cm{
&&\hspace{-35pt} \mathcal{J}_{n, \pi}^{\text {cris }}(G^{\prime}) &&\\
\hspace{10pt}\mathcal{J}_{n, \pi}^{\text {cris }}(G) \ar@{}[ur]^(.3){}="a"^(.8){}="b" \ar^{\psi} "a";"b" && & & \\
&&\hspace{-15pt} \mathcal{J}_{n+1, \pi}^{\text {cris }}(G^{\prime}) \ar@{}[r]^(.4){}="a"^(1.1){}="b" \ar^{\phi} "a";"b" \ar[]!<-5ex,-8ex>;[uu]!<-5ex,-3ex>&& \hspace{-25pt}\mathcal{O}_{n+1, \pi}^{\text {cris }}(G^{\prime})\\
\hspace{5pt}\mathcal{J}_{n+1, \pi}^{\text {cris }}(G) \ar[uu] \ar@{}[ur]^(.3){}="a"^(.8){}="b" \ar^{\psi} "a";"b"\ar@{}[rr]^(.5){}="a"^(1.2){}="b" \ar^{\phi} "a";"b" && & \hspace{-5pt} \mathcal{O}_{n+1, \pi}^{\text {cris }}(G) \ar@{}[ur]^(.3){}="a"^(.8){}="b" \ar_{\psi} "a";"b"&
}
$$

We start from picking any element $\alpha \in \mathcal{J}_{n, \pi}^{\text {cris }}(G)$. Suppose $\beta \in \mathcal{J}_{n+1, \pi}^{\text {cris }}(G)$ lifts $\alpha$. Applying $\psi$, we get $\psi(\beta) \in \mathcal{J}_{n+1, \pi}^{\text {cris }}(G^{\prime})$, then $\psi(\beta)$ also lifts $\psi(\alpha).$ Applying the Frobenius map $\phi \colon \mathcal{O}_{n, \pi}^{\operatorname{cris}} \to \mathcal{O}_{n, \pi}^{\operatorname{cris}} $ induced by the crystalline Frobenius, we get $\phi(\psi(\beta))=\psi(\phi(\alpha))$ in $\mathcal{O}_{n+1, \pi}^{\text {cris }}(G)$.

By Corollary \ref{Breuil 2.3.3} (see also \cite[Lemma 2.3.2]{BR00}), we deduce that $\phi\left(J_ {n, \pi}^{c r i s}\right) \subset p \mathcal{O}_{n, \pi}^{cris}$. To define the divided Frobenius $\phi_1$ (sheaf theoretically), we need to locally lift to modulo $p^{n+1}$, then apply $\phi$, and reduce back to modulo $p^n$ by dividing by $p$, as in \cite[II.2.3]{FM}. By definition, we have 
\begin{align}
    \phi(\beta)&=p(\phi_1(\alpha)) \\
    \phi(\psi(\beta))&=p\phi_1(\psi(\alpha)).
\end{align}
Therefore 
$$\psi(\phi(\beta))=p\psi(\phi_1(\alpha))$$
and 
$$\psi(\phi_1(\alpha))=\phi_1(\psi(\alpha)) $$ in $\mathcal{O}_{n, \pi}^{\text {cris }}(G).$ 
\end{itemize}

 \end{proof}

\subsection{p-adic neighborhoods of large monodromy}
We are now ready to prove a criterion for a $p$-divisible group $G/\ok$ to have large monodromy in the sense of Definition \ref{defnlarge monodromy} by proving the following theorem.

\begin{prop}\label{open neighborhoods}
	Let $G$ be the (unramified) connected reductive group that defines the Shimura datum. Let $$K_n=\{M \in \Sp_{2g}(\Z_{p}) \mid M \equiv I \text{ mod } p^{n} \}$$ for some fixed $n \ge 1$. Since $K_n$ is open in $\Sp_{2g}$, the intersection $H_n=G^{der}(\Z_p)\cap K_n$ is open in $G^{der}(\Q_p)$. Let $\mathcal{G}$ be a $p$-divisible group over $\ok$ and suppose that its $p$-adic Galois image contains $H_n$. Let $\mathcal{G}^{\prime}$ be another $p$-divisible group over $\ok$ that satisfies $\mathcal{G}[p^{n+1}] \simeq \mathcal{G}^{\prime}[p^{n+1}]$ as group schemes over $\ok$. Then the $p$-adic Galois image of $\mathcal{G}^{\prime}$ also contains $H_m$ for some $m \ge n$.
\end{prop}
 
\begin{proof}
Denote by $\rho_{\G}$(resp. $\rho_{\G^{\prime}}$) the Galois representation associated with $\G$ (resp. $\G^{\prime}$). The assumption $\G[p^{n+1}] \simeq \G^{\prime}[p^{n+1}]$ implies that the image of $\rho_{\G}$ is isomorphic to that of $\rho_{\G^{\prime}}$ modulo $p^{n+1}$. Thus we have $H_n \bmod p^{n+1} \subset \Img(\rho_{\G} \bmod p^{n+1})$. 

A matrix $M \in H_n$ can be written as formal power series $M=I+\sum_{i=n}^{\infty} p^{i}M_{i}$ with each $M_{i} \in M_{2}(\F_{p})$. In particular, for $i \ge n$, let $L_i$ be the set of matrix in $ \Sp_{2g}(\Z/p^{i+1}\Z) $ such that 
$$L_i \coloneqq H_{i}/H_{i+1} \bmod p^{i+1}.$$
By definition, we have $L_{\infty}=H_n.$
The assumption and the preceding argument indicate that the image of $\rho_{\G^{\prime}}$ modulo $p^{n+1}$ contains $L_n$. Raising $L_n$ to the $p$-th power, we claim that the Galois image of $\rho_{\G^{\prime}}$ modulo $p^{n+2}$ also contains the set $L_{n+1}$. By induction, we conclude that $\Img(\rho_{\G^{\prime}})$ contains $H_n$.

To prove the claim, we show that the $p$-power map
\begin{align}\label{p-power}
    H_m/H_{m+1} &\to H_{m+1}/H_{m+2} \\
    A &\to A^p 
\end{align}
is surjective for some $m \ge n$. In general, for any linear algebraic group $G$ there exists such $m$. 

Algebraic groups over $p$-adic fields give rise to $p$-adic Lie groups, and the associated Lie algebra is the Lie algebra of the initial algebraic group. There is a notion of the $p$-adic exponential map similar to the algebraic group setting, which is only locally defined: there is a compact open subring of the Lie algebra $\operatorname{Lie(G)}$ on which the exponential is defined (see \cite{Bourbaki} for a reference). The exponential map is always injective and its image is a compact open subgroup of $G$, i.e., it contains $H_m$ for all $m \ge k$ where $k$ is a positive integer that depends only on $G$. Suppose $G=\operatorname{GL}_{2g}$. The Lie algebra of $G(\Q_p)$ is $M_{2g \times 2g}(\Q_p)$ and the inverse image of $H_m$ is $M_{2g \times 2g}(p^{m}\Z_p)$. Then the $p$-power map in the Lie group becomes multiplication by $p$ map in the Lie algebra
$$M_{2g \times 2g}(p^{m}\Z_p)/M_{2g \times 2g}(p^{m+1}\Z_p) \to M_{2g \times 2g}(p^{m+1}\Z_p)/M_{2g \times 2g}(p^{m+2}\Z_p)$$ which is an isomorphism. This proves the surjectivity of (\ref{p-power}). The claim follows from the fact that the subgroup exponential map $G^{\prime} \subset G$ of a Lie subgroup $G^{\prime}$ is the restriction of the exponential map of $G$, due to the naturality of the exponentials. 

\end{proof}

We conclude that Theorem \ref{congruence condition} follows from Lemma \ref{breuil lemma} and Proposition \ref{open neighborhoods}.

\section{Points with large monodromy are dense}
We make recollections on the Rapoport-Zink spaces of Hodge type, the $p$-adic uniformizations and the Grothendieck-Messing period map. We follow the set-up of \cite{kis} and also \cite{KIM_Uniformization}, to which the reader is referred for more details.
\subsection{Shimura varieties of Hodge type and the local Shimura-Hodge datum}\label{localdatum}
Let $(G, X)$ be a Shimura datum in the sense of Deligne's formalism, where $G$ is a connected reductive group over $\mathbb{Q}$ and let $X=\{h\}$ be the $G(\mathbb{R})$-conjugacy class of a Deligne cocharacter $h: \operatorname{Res}_{\mathbb{C} / \mathbb{R}} \mathbb{G}_m \rightarrow G_{\mathbb{R}}$. Let $\mu$ be the minuscule cocharacter $\mu_h: \mathbb{G}_{m \mathbb{C}} \rightarrow G_{\mathbb{C}}$ defined by $\mu_h(z)=h_{\mathbb{C}}(z, 1)$. Let $E \subset \overline{\mathbb{Q}} \subset \mathbb{C}$ be the reflex field that is the field of definition of the conjugacy class $\left\{\mu_h\right\}$. We call a Shimura datum is of \textit{Hodge type} if there is an embedding of Shimura data $$(G, X) \rightarrow\left(\mathrm{GSp}_{2 g}, \mathcal{S}^{\pm}\right)$$ induced by an embedding of algebraic groups

$$
\iota: G \hookrightarrow \mathrm{GSp}_{2 g}
$$over $\mathbb{Q}$. Here $\mathcal{S}^{\pm}$ is the union of the Siegel upper and lower half spaces of genus $g$. A choice of field embedding $\overline{\mathbb{Q}} \hookrightarrow \overline{K}_0$ determines a place $v \mid p$ of $E$ and the completion $E_v$ is the field of definition of $\left\{\mu_h\right\}$, regarding as a $G(\overline{K_0})$-conjugacy class of cocharacters $$\mu : \mathbb{G}_{m \overline{K_0}} \rightarrow G_{\overline{K_0}}.$$

After the work of Kisin \cite{kis} and Vasiu \cite{Vas1,Vas2,Vas3}, there is a good theory of canonical integral models of Shimura varieties of Hodge type at primes of good reduction. For $G$ unramified, let $K=K_p K^p \subset G\left(\mathbb{A}_f\right)$ be a sufficiently small compact open subgroup of $ G\left(\mathbb{A}_f\right)$ with $K_p \subset G\left(\mathbb{Q}_p\right)$ hyperspecial. The Shimura variety $\mathcal{S}h_{K}(G, X)$ admits a canonical smooth integral model$$
\mathscr{S}=\mathscr{S}_K(G, X)
$$over the localization $\mathcal{O}_{E,(v)}$. 

A point $x_0 \in \mathscr{S}(k)$ on a global Shimura variety of Hodge type gives rise to a  \textit{local Shimura-Hodge datum} $\left(G_{\mathbb{Z}_p}, [b_{x_0}], \{\mu_{x_0}\}, C_{\mathbb{Z}_p}\right)$. Let $X_0=A_{x_0}\left[p^{\infty}\right]$ be the $p$-divisible group of the abelian variety $A_{x_0}$ over $k$ determined by $x_0$ and let $\mathbb{D}\left(X_0\right)(W)$ be the \text {Dieudonné module} of $X_0$. By \cite[Cor. 1.4.3]{kis}, there are crystalline tensors
$$
t_{\alpha, 0} \in \mathbb{D}\left(X_0\right)(W)^{\otimes}
$$
which are Frobenius invariant in $\mathbb{D}\left(X_0\right)(W)^{\otimes}[1 / p]$. Moreover, there is a $W$-module isomorphism
$$
D \otimes_{\mathbb{Z}_p} W \xrightarrow{\sim} \mathbb{D}\left(X_0\right)(W)
$$identifying $s_\alpha \otimes 1$ with $t_{\alpha, 0}$ where $s_\alpha$ are tensors that cut out the group $G$. After fixing the isomorphism, the Frobenius operator on $\mathbb{D}\left(X_0\right)(W)$ induces an operator on $D \otimes_{\mathbb{Z}_p} W$ of the form $F=b_{x_0} \circ \sigma
$ for some $b_{x_0} \in G(K_0)$, where $\sigma \in \operatorname{Aut}(W)$ lifts the absolute Frobenius on $k$. By [loc. cit.], the Hodge filtration on $\mathbb{D}\left(X_0\right)(k)$ is given by a $G_k$-valued cocharacter. We may pick any lift to a cocharacter $\mu_{x_0}: \mathbb{G}_{m W} \rightarrow G_W$ such that $$b_{x_0} \in G(W) \mu_{x_0}^\sigma(p) G(W) . $$

The $G(W)$-conjugacy class of $\mu_{x_0}$ is the same as the conjugacy class of the inverse of the Deligne cocharacter $\mu_h$. In particular, $\mu_{x_0}$ and $\mu_h^{-1}$ become conjugate after fixing an isomorphism $\mathbb{C} \xrightarrow{\sim} K_0$ whose restriction to $E \hookrightarrow \overline{K}_0$ induces the place $v$.

Fix $b:=b_{x_0} \in[b_{x_0}]$. Associated to the local Shimura datum is an algebraic group $J_{b}$ over $\Q_p$ such that $$J_b(\Q_p):=\left\{g \in G\left(K_0\right) \mid g b \sigma(g)^{-1}=b\right\}.$$ It is the group of quasi-isogenies that preserves the crystalline tensors. There is a natural right action of $J_b (\mathbb{Q}_p)$ on the formal scheme $\mathrm{RZ}_{G, b}$ defined as follows. Let $R \in \operatorname{Nilp_W}$ and let $(X, \varrho)$ be a point in $\mathrm{RZ}_{G, b}(R).$ Let $\gamma \in J_b(\Q_p),$ then 

$$
(X, \varrho, (t_\alpha)) \longmapsto(X, \varrho \circ \gamma, (t_\alpha)).
$$

\subsection{the Rapoport-Zink spaces of Hodge type}
Just like Hodge type Shimura varieties are constructed from Hodge type Shimura data, there is a natural $p$-adic local analogue of such Shimura varieties arising from the local Shimura-Hodge datum serving as the “moduli spaces” of $p$-divisible groups equipped with certain crystalline Tate tensors. In the work of Howard and Pappas \cite{Howard_Pappas_2017}, an analogue of the Rapoport-Zink space is constructed for the local Shimura datum arises from a point on the integral model of a global Shimura variety of Hodge type, generalizing the work of \cite{RZ}. The construction is also done in the work of Kim \cite{KIM_RZ}, and Kim's method works in a more general setting that does not require the local Shimura datum coming from points on a global Shimura variety. More precisely, let $\mathbb{X}$ be a p-divisible group over $\overline{\F}_p.$ There exists a closed formally smooth subscheme $\mathrm{RZ}_{G, b} \subset \mathrm{RZ}_{\mathbb{X}}$ that only depends on the associated unramified Hodge-type local Shimura datum $\left(G,[b],\left\{\mu^{-1}\right\}\right)$, which classifies deformations of $\mathbb{X}$ with Tate tensors $\left(\mathbf{t}_\alpha\right)$ up to quasi-isogeny, such that the Hodge filtration of the $p$-divisible group is \'etale-locally given by some cocharacter in the conjugacy class $\{\mu\}$ \cite{KIM_RZ}. Then the local Shimura varieties are the rigid analytic tower $\{\mathrm{RZ}_{G, b}^{K}\}$ over the generic fiber of $\mathrm{RZ}_{G, b}$ when $K \subset G\left(\mathbb{Q}_p\right)$ runs through open compact subgroups of $\left.G\left(\mathbb{Q}_p\right)\right)$.

\subsection{The Grothendieck-Messing period map} Since $\mathrm{RZ}_{G, b}$ is locally formally of finite type over $\operatorname{Spf} W$, Berthelot's construction of rigid generic fiber $\mathrm{RZ}_{G, b}^{\text{rig }}$ can be applied to obtain the period map on the rigid generic fiber $\mathrm{RZ}_{G, b}^{\text{rig }}$ which is \'etale.

 For a given local Shimura datum $(G, b, \{\mu\})$, one can associate a flag variety $\mathcal{F}:=\mathcal{F}(G,\{\mu\})$ to the parabolic groups of type $\mu$, which is an algebraic variety over $E$. Let $\breve{E}=\breve{\Q}_{p}. E$ be the completion of the maximal unramified extension of $E \subset \overline{\Q}_p$. We denote by $$\breve{\mathcal{F}}=\mathcal{F} \times_{\Spec E} \Spec{ \breve{E}}$$ which is an algebraic variety over $\breve{E}.$ Let $\breve{\mathcal{F}}^{\text{rig}}$ be the rigid analytic space associated with $\breve{\mathcal{F}}$ over $\breve{E}.$ Then we have the period map
$$
\breve{\pi} : \mathrm{RZ}_{G, b}^{\text{rig }} \rightarrow \mathcal{F}^{\text{rig}},
$$
which is an \'etale morphism and is $J_b\left(\mathbb{Q}_p\right)$-equivariant. 

    \subsection{The connectedness of p-adic period domains} Denote by $$\breve{\mathcal{F}}^{\text{wa}}=\breve{\mathcal{F}}(G, b,\{\mu\})^{\text{wa}}$$ the \textit{period domain} associated with $(G, b,\{\mu\})$. It is an admissible open subset of $\breve{\mathcal{F}}^{\text {rig }}$ consists of elements in $\breve{\mathcal{F}}(\overline{\Q}_p)$ such that $(b, \mu)$ is a weakly admissible pair for $G$. Let $\breve{\mathcal{F}}^{\text{a}}$ be the image of $\breve{\pi}$ as an adic space. Then $\breve{\mathcal{F}}^{\mathrm{a}} \subseteq \breve{\mathcal{F}}^{\text{wa}}$. We will use the fact that their points as adic spaces have the same sets of points with values in a finite extension of $\breve{E}$ \cite{CF00}.

It is conjectured that $\breve{\mathcal{F}}^{\text{a}}$ is arcwise connected \cite[Conjecture 6.5]{Hartl}. This fact was proved in a series of articles in different settings \cite{CH14,Gle,GLX}, and in complete generality in \cite[Theorem 1.1]{GL}. We prove the following lemma, which was conjectured in \cite{RV} and proved in the case of unramified EL and PEL type in \cite{CH14}.
\begin{lem}\label{connectedness}
    For any connected component $M$ of $\mathrm{RZ}^{\text{rig}}_{G,b}$, $\breve{\pi}(M)=\breve{\mathcal{F}}^{\text{a}}$.
\end{lem}
\begin{proof}
    Let $\left\{\mathrm{RZ}_{G, b}^K\right\}$ be the $G\left(\mathbb{Q}_p\right)$-equivariant tower of finite \'etale covers of $\mathrm{RZ}_{G, b}^{\text {rig}}$ constructed from open compact subgroups $K \subset G\left(\mathbb{Z}_p\right)$ in \cite[Section 7.4]{KIM_RZ}. The right $G(\Q_p)$ action on the tower is given by a collection of isomorphisms
    $$[g]: \mathrm{RZ}_{G, b}^{K_g} \xrightarrow{\sim} \mathrm{RZ}_{G, b}^{gK_{g}g^{-1}}$$ where $K_g :=gKg^{-1} \cap K$ (\cite[Proposition 7.4.8]{KIM_RZ}).
    We have the following commutative diagram:
    \begin{equation*}
\begin{tikzcd}
   \mathrm{RZ}_{G, b}^{K_{g}} \arrow[rr, "g"] \arrow[dr, "\pi_{K_{g}}", swap] & &  \mathrm{RZ}_{G, b}^{gK_{g}g^{-1}} \arrow[dl, "\pi_{gK_{g}g^{-1}}"] \\
   &  \mathrm{RZ}_{G, b}^{\text{rig}} \arrow[d, "\breve{\pi}"] & \\
   & \breve{\mathcal{F}}^{\text{a}} & 
\end{tikzcd}
\end{equation*}
where $\pi_{K_{g}}$ and $\pi_{gK_{g}g^{-1}}$ are projection maps.
Gleason and Lourenco proved that $p$-adic period domains are geometrically connected \cite{GL}. In particular, their Theorem 1.1 implies that $\breve{\mathcal{F}}^{\text{a}}$ is connected. In \cite[Proposition 4.8, Lemma 4.9]{Gle}, Gleason also proved that a period domain is connected is equivalent to the group $G(\Q_p)$ acting transitively on the infinite-level Rapoport-Zink spaces $ \mathrm{RZ}_{G, b}^{\infty}$. We claim that $G(\Q_p)$ acting transitively on connected components of the tower is equivalent to acting transitively on $ \mathrm{RZ}_{G, b}^{\infty}$. Indeed, for any $K_1 \subset K_2$, the transition map $$ \mathrm{RZ}_{G, b}^{K_1} \to  \mathrm{RZ}_{G, b}^{K_2}$$
is a surjective map on the connected components between the underlying topological spaces of the source and the target. Therefore, the limit of $\pi_0$ is that of the limit
$$\pi_0(\mathrm{RZ}_{G, b}^{\infty}) = {\varprojlim_K } \text{ } \pi_0(\mathrm{RZ}_{G, b}^{K}) .$$
This together with the fact that the underlying topological space of the diamond associated to a rigid analytic space $\mathrm{RZ}_{G, b}^{K}$ is the same as $\left| \mathrm{RZ}_{G, b}^{K} \right|$ so that they have the same connected components, one deduce that the $G(\Q_p)$ action is transitive on the connected components of rigid analytic spaces. In particular, all the connected components of $\mathrm{RZ}_{G, b}^{\text {rig}}$ have the same image under the map $\breve{\pi}$. 
\end{proof}

\subsection{The rigid analytic uniformizations of Shimura varieties} The completion of the integral model $\mathscr{S}$ along an isogeny class can be described by \text{p-adic uniformizations}. When the Shimura variety is of PEL type, Rapoport and Zink constructed the uniformization in \cite{RZ}, where the uniformization is a formal scheme and the basic locus of $\mathcal{S}h_{K}(\overline {\F}_p)$ can be interpreted as a quotient of it. Their results were generalized to Hodge type Shimura varieties by Kim \cite{KIM_Uniformization} and Howard-Pappas \cite{Howard_Pappas_2017}. The $p$-adic uniformization theorem implies a corresponding rigid analytic uniformization theorem that allows us to represent an admissible open subset of the rigid analytic variety $\mathcal{S}h_{K}^{\text{rig}}$ associated to the Shimura variety $\mathcal{S}h_{K}(G, X)$ as direct sum of quotients of $\mathrm{RZ}_{G, b}^{\text{rig }}$. In particular, let $\phi$ be an $\overline{\F}_p$-isogeny class which is a countable (finite if $\phi$ is in the basic Newton stratum) union $\mathscr{I}^\phi:=\left\{Z_i\right\}_{i}$ of irreducible closed subsets (resp. irreducible subvarieties) of $\mathscr{S}_{\overline{\mathbb{F}}_p}$ and denote by $$\mathcal{S}h_{K}^{\mathrm{rig}}(\phi):=\left(\left(\widehat{\mathscr{S}}_{W}\right)_{/ \mathscr{I} \phi}\right)^{\mathrm{rig}}$$ the union of tubular neighborhood of the irreducible subvarieties of $\mathscr{I}^\phi$, then one can interpret $\mathcal{S}h_{K}^{\text {rig }}(\phi)$ as the (finite) disjoint sum of quotients of its coverings by a discrete subgroup of $J_b(\Q_p)$ \cite[Theorem 4.7; Theorem 5.4]{KIM_Uniformization}. We summarize the results of uniformization in the following theorem.

\begin{thm}\label{uniformization}
    Let $\mathcal{S}h_{K,0}^{\mathrm{rig}}(\phi)$ be a connected component of $\mathcal{S}h_{K}^{\mathrm{rig}}(\phi).$ For sufficiently small $K^p$, there is a connected component $\mathrm{RZ}_{G, b}^{\mathrm{rig},0}$ of $\mathrm{RZ}_{G, b}^{\mathrm{rig}}$ such that the restriction map $\Theta^{\phi, \mathrm{rig}}\colon \mathrm{RZ}_{G, b}^{\mathrm{rig},0} \to \mathcal{S}h_{K,0}^{\mathrm{rig}}(\phi)$ is a topological covering of $\mathcal{S}h_{K,0}^{\mathrm{rig}}(\phi)$. Specifically, $\mathcal{S}h_{K,0}^{\mathrm{rig}}(\phi)$ is the quotient of $\mathrm{RZ}_{G, b}^{\mathrm{rig},0}$ by a discrete subgroup of $J_b(\Q_p).$
\end{thm}
\begin{remark}
    When $\phi$ corresponds to some $b$ in the basic locus, then $\mathcal{S}h_{K}^{\mathrm{rig}}(\phi)$ is the tube of $\mathscr{I}^\phi$ in $\widehat{\mathscr{S}}_{K, W}$. In general, $\mathcal{S}h_{K}^{\text {rig }}(\phi)$ is a union of tubes of the irreducible subvarieties $Z \in \mathscr{I}^\phi$ \cite[5.3]{KIM_Uniformization}. For a $p$-divisible group $\mathbb{X}$ in the isogeny class $\phi$, let $X$ be a lifting of $\mathbb{X}$ over some finite extension $\mathcal{O}_K$ of $\mathcal{O}_E$, then $X$ is contained in one of the tubes in $\mathcal{S}h_{K}^{\text {rig }}(\phi)$.
\end{remark}

\subsection{Points with large monodromy are dense}
In this section we show that the locus with large monodromy $\mathcal{S}^{K}_{max}$ is dense in $\mathcal{S}^K$ by showing that for a fixed formal lifting $\tilde{A}$ of $A \in \mathcal{S}h_{K}(\overline{\F}_p)$, there exists a generic point in any $p$-adic neighborhood of $\tilde{A}$.

For a finite extension $\overline{\bE} / K / \bE$, the $K$-points in the weakly admmisible locus $\mathcal{F}^{\text{wa}}$ are admissible by Colmez-Fontaine \cite{CF00}. Therefore, for each point $x$ in $\breve{\mathcal{F}}^{\text{wa}}(K)$, we get a crystalline representation with additional structures attached to $x$
$$\xi_{x}: \operatorname{Gal}(\bar{\bE}/K) \to G(\Q_{p})$$
which is well defined up to $G(\Q_{p})$-conjugation. Let $X^{univ}_{G, b}$ be the universal $p$-divisible group on $\mathrm{RZ}_{G, b}$. For $x_0 \in \mathrm{RZ}_{G, b}^{\text{rig}}(K)$ such that $\breve{\pi}(x_0)=x$, there is a Galois representation attached to the Tate module 
 $$\operatorname{Gal}(\bar{\bE}/K) \to V_{p}(X_{x_0,G,b}^{univ}).$$ Thus, we obtain the crystalline representation $\xi_{x_0}$ after fixing an isomorphism $$V \simeq V_p(X_{x_0, G, b}^{univ})$$ where $(V, \psi )$ is the symplectic space $$G \hookrightarrow \operatorname{GSp}(V, \psi).$$

The concept of the \textit{generic point} is defined in the sense of \cite[Chapter 5]{CH14}.
\begin{defn}
Let $K$ be a finite extension of $\bE$. Let $R$ be an $E$-algebra with transcendental generators $a_{1},\cdots,a_{n}$ such that $R$ is a finite extension of $\Q_{p}[a_{1},\cdots,a_{n}]$. Write $\mathcal{F}$ in the affine form $\mathcal{F}= \operatorname{Spec} R$. A point $x \in \mathcal{F}(K)$ is called \textit{generic} if the corresponding homomorphism $x: R \to K$ sends $a_{1},\cdots,a_{n}$ to a set of elements $b_{1},\cdots,b_{n}$ that are algebraically independent over $\Q_{p}$.
\end{defn}

The following theorem indicates that a generic point has large monodromy, which is a key ingredient in our proof of proposition \ref{dense}. 

\begin{thm}\cite[Theorem C]{CH14}\cite[Theorem 1.14]{GLX} 
  Suppose that the Hodge and Newton polygons associated with $\{\mu\}$ and $[b]$ do not touch outside their extremities. Assume $[b] \in B(G, {\mu})$. Let $K/ \bE$ an extension of finite degree and $x \in \mathcal{F}(K)$ which is a generic point of $\mathcal{F}$. Then $x \in \bF^{\text{wa}}(K)$ and the image of the crystalline representation $$\xi_{x}\colon \Gal(\bar{\bE}/K) \to G(\Q_{p})$$ contains an open subgroup of $G^{der}(\Q_{p})$.
\end{thm}

\vspace{1em}

The property of being generic spreads "nicely" in the sense that for a moduli point $x \in \mathrm{RZ}^{\text{rig}}_{G,b}(\ok)$, there exist generic points in any neighborhood of $x$ up to possibly replace $K$ with a finite extension. We prove this by proving that the $\hat{\Q}_{p}^{un}$ points of $\breve{\mathcal{F}}^{\text{a}}$ are well-approximated by generic points.

\begin{lem}\label{approximation}
Given $a_{1}, \cdots, a_{n} \in K_0 $ and $r \ge 0$. There exists $b_{1}, \cdots, b_{n} \in K_0$ such that

\begin{itemize}
\item $b_{1}, \cdots, b_{n}$ are algebraically independent over $\Q_{p}$
\item For each $1 \le i \le n$, $|b_{i}-a_{i}|_{p} < p^{-r}$.	
\end{itemize}
  
\end{lem}
 \begin{proof}
 	Since $\Q_{p}^{un}$ is dense in $\hat{\Q}_{p}^{un}$ we can assuming that $a_{1}, \cdots, a_{n} \in \Q_{p}^{un}$ for simplicity. Especially they are algebraic over $\Q_{p}$. Therefore it is enough to prove the case for $a=a_{1}= \cdots= a_{n}$ and the lemma then follows from shifting $b_{i}$ to $b_{i}+a_{i}-a_{1}$ for $i \ge 2$.
 	
 	Denote by $U$ the set of all roots of unity whose orders are coprime to $p$. It is a standard fact that every element of $\hat{\Q}_{p}^{un}$ is uniquely represented as $\sum_{n \ge k} \zeta_{n} p^{n}$ where $\zeta_{n} \in U$ and $k \in \Z$. Let $a=\sum_{n \ge k} \zeta_{n,a} p^{n}$. The $p$-adic neighborhood $$B(a,r):=\{z \in \hat{\Q}_{p}^{un} \mid |z-a|_{p} < p^{-r}\}$$
 	consists of elements of the form $\{\sum_{n=k}^{r}\zeta_{n,a}p^{n} + \sum_{n \ge r+1}\zeta_{n}p^{n}\}$.
 	By \cite[Example 1]{Nis}, for $n \ge r+1$, take  
 	$$\zeta_{n}=\zeta(\ell^{[\lambda n]}), \text{    ($\ell \in P-{p}$, $\lambda \in \R^{+}$)}$$
 	and let $\lambda$ varying in $\R^{+}$, there are uncountably infinitely many algebraically independent transcendental elements inside $B(a,r)$. We conclude the proof by choosing $b_{1},\cdots,b_{n}$ as any $n$ of them. 
 \end{proof}

\begin{prop}\label{dense}
Fix a point of $\mathcal{S}h_{K}(\overline{\F}_p)$ and let $A$ be the associated abelian variety over $\overline{\F}_p$. Let $K/\breve{E}$ be a finite extension, and let $\tilde{A}$ be a lifting of $A$ over $\ok$. Then the set of $\ok$-liftings with large monodromy $\mathcal{S}^{K,A}_{max}$ is dense in $\mathcal{S}^{K,A}$ after possibly replacing $K$ with a finite extension. 
\end{prop}
\begin{proof}
Consider the following diagram:
$$
\begin{tikzcd}
   \mathrm{RZ}_{G, b}^{\textup{rig},0}(K) \arrow[r, "\breve{\pi}"] \arrow[d, swap] & \breve{\mathcal{F}}^{\text{a}}(K) & \\
  \Gamma\backslash \mathrm{RZ}_{G, b}^{\text{rig},0}(K) \arrow[r, "\Theta^{\phi, \mathrm{rig}}"] & \mathcal{S}h_{K,0}^{\text{rig}}(\phi)(K) & \hspace{-35pt}\xhookleftarrow{} \mathcal{S}h_{K,0}^{\text{rig}}(\phi)(\ok)
\end{tikzcd}
$$
	 Take a point $x \in \mathcal{S}^{K,A}$ so that its base change to $K$ is contained in some connected component $\mathcal{S}h_{K,0}^{\mathrm{rig}}(\phi)(K)$. We slightly abuse the notation and also call the $K$-point $x$. Let $U_x$ be a $p$-adic open neighbourhood of $x$. By Theorem \ref{uniformization}, $x$ corresponds to an $\Gamma$-conjugacy class of points $[(X, \rho,(t_\alpha))]$ where $\Gamma$ is a discrete subgroup of $J_b(\Q_p)$. Let $y=(X, \rho, (t_\alpha)) \in \mathrm{RZ}^{\mathrm{rig},0}_{G,b}(K)$ be a representative of this conjugacy class. Note that $X$ is a $p$-divisible group over $K$ that is the same as the $p$-divisible group associated with $x$. Let $V_{y}$ be the preimage of $U_x$ under $\Theta^{\phi, \mathrm{rig}}$. 
     
    Let $W_{y}=\breve{\pi}(V_{y})$ be its image under $\breve{\pi}$. Let $x_0=\breve{\pi}(y)$ be the image of $y$ in $\breve{\mathcal{F}}^{\text{a}}(K)$. According to Lemma \ref{approximation}, there exists a generic point $z \in \breve{\mathcal{F}}^{\text{a}}(K) $ such that $z \in  W_{y}$. The connectedness of the period domain (Lemma \ref{connectedness}) indicates that the fiber $\breve{\pi}^{-1}(z)$ is non-empty in each connected component of $\mathrm{RZ}^{\text{rig}}_{G,b}$. Since $\breve{\pi}^{-1}(z)$ is \'etale over $K$, any connected component of it is a finite extension $K'$ of $K$. After replacing $K$ with $K'$, we obtain a $K$-point $y_0 \in V_{y} \cap \breve{\pi}^{-1}(z)$ that is given by a $\ok$-point on $\mathrm{RZ}^{0}_{G,b}$. This point also has large monodromy and maps to a $\ok$ point $z \in V_x$ under the uniformization map that also has large monodromy.

\end{proof}

\begin{proof}[Proof of Theorem \ref{padic}]
    The theorem follows from Theorem \ref{congruence condition} and Proposition \ref{dense}.
\end{proof}

\newpage
\bibliography{reference}
\end{document}